\documentclass[10pt, a4paper]{amsart}
\usepackage{amsmath,amstext,amscd, amssymb,txfonts, epsfig, psfrag, color, multicol, graphicx}
\usepackage[normalem]{ulem}
\usepackage[colorlinks=true, pdfstartview=FitV, linkcolor=blue, citecolor=blue, urlcolor=blue]{hyperref}

\theoremstyle{plain}

\newtheorem{thm}{Theorem}[section]
\newtheorem{lemma}[thm]{Lemma}
\newtheorem{cor}[thm]{Corollary}
\newtheorem{prop}[thm]{Proposition}

\theoremstyle{definition}

\newtheorem{Open questions}[thm]{Open questions}
\newtheorem{Open question}[thm]{Open question}
\newtheorem{Open problems}[thm]{Open problems}
\newtheorem{Open problem}[thm]{Open problem}

\newcommand{\tc}[2]{\textcolor{#1}{#2}}
\definecolor{dmagenta}{rgb}{.5,0,.5} 
\definecolor{dred}{rgb}{.5,0,0} 
\definecolor{dgreen}{rgb}{0,.5,0} 
\definecolor{dblue}{rgb}{0,0,0.5} 
\definecolor{black}{rgb}{0,0,0} 
\definecolor{vdgreen}{rgb}{0,.3,0} 
\definecolor{vdred}{rgb}{.3,0,0} 
\definecolor{red}{rgb}{1,0,0} \newcommand{\red}[1]{\tc{red}{#1}}
\definecolor{gray}{rgb}{.5,.5,.5}

\def\Bbb{\mathbb}
\def\bar{\overline}

\def\Z{\Bbb{Z}}
\def\N{\Bbb{N}}

\def\Reals{\Bbb{R}}

\def\ni{\noindent}

\def\Dist{\hbox{\rm Dist}}

\def\F+L{\hbox{$\textup{F}\!_+\textup{L}$}}

\renewcommand{\H}{\mathcal{H}}

\def\ms{\medskip}

\def\onto{{\kern3pt\to\kern-8pt\to\kern3pt}}

\def\<{\langle}
\def\>{\rangle}
\def\|{{\ |\ }}

\newcommand{\set}[1]{\left\{#1\right\}}

\renewcommand{\ni}{\noindent}

\renewcommand{\ms}{\medskip}
\newcommand{\bs}{\bigskip}
\def\*{^{\star}}

\begin{document}

\title{Hyperbolic hydra}

\author{N.\ Brady,  W.\ Dison  and T.R.\ Riley}

\thanks{The first author acknowledges partially supported from NSF grant DMS--0906962.}

\date \today

\begin{abstract}
\ni We give examples of hyperbolic groups with finite--rank free subgroups of huge (Ackermannian) distortion.   
 
 \ms

\footnotesize{\ni \textbf{2010 Mathematics Subject
Classification:  20F65, 20F10, 20F67}  \\ \ni \emph{Key words and phrases:} hyperbolic group, subgroup distortion,  hydra, Ackermann's function}
\end{abstract}

\maketitle

\section{Introduction} \label{intro}

\subsection{Our result.} Hyperbolic groups  are algorithmically tractable (their word and conjugacy problem are straight--forward) and have the tree--like property that geodesic triangles in their  Cayley graphs are close to tripods  \cite{BrH, Gromov4}.  The purpose of  this article is to show that none--the--less some harbour extreme wildness  within their subgroups ---  their finite--rank free subgroups, even.   We prove (the terminology is explained below) ---
 

\begin{thm} \label{Main}
For all $k \geq 1$ there is a hyperbolic group $\Gamma_k$ with a free  rank--$(k+18)$  subgroup $\Lambda_k$  whose distortion is   $\Dist^{\Gamma_k}_{\Lambda_k}  \succeq  A_k$ --- that is, it grows at least as fast as the $k$--th of the family of Ackermann functions.  
 \end{thm}

A \emph{distortion function} measures the degree to which a subgroup  folds in on itself within the ambient group.  Suppose $H$ is a subgroup with  generating set $T$ inside a group $G$ with finite generating set $S$.  Then $\Dist^{G}_{H} : \N \to \N$ compares the intrinsic word metric $d_T$ on $H$ with the extrinsic word metric $d_S$: 
$$\Dist^{G}_{H}(n) \ := \  \max \set{  \  d_T(1,g)  \  \mid  \   g \in {H} \textup{ with }  d_S(1,g)  \leq n   \ }.$$
Up to the following equivalence, capturing qualitative agreement of growth rates, $\Dist^{G}_{H}$ does not depend on $S$ and $T$.   For $f,g:\N \to \N$, we write $f \preceq g$ when there exists $C>0$ such that $f(n) \leq Cg(Cn+C)+Cn+C$ for all $n$.  Defining $f \simeq g$ when $f \preceq g$ and $g \preceq f$.  

 Ackermann's $A_k  : \N \to \N$ are a family of fast--growing functions defined recursively:  
\begin{align*}
A_0(n) & \ = \ n+2    \textup{ for }    n \geq 0,   \\
A_k(0) & \ =  \
\begin{cases}
0 & \textup{ for }   k = 1  \\
1 & \textup{ for }   k \geq 2,
\end{cases}
     \\
  \ \ \ \text{ and }  \ \ \
A_{k+1}(n+1) & \  =  \  A_k({A_{k+1}}(n))  \textup{ for }   k,n  \geq 0.
\end{align*}
In particular, $A_1(n)=2n$,  $A_2(n)=2^n$ and $A_3(n)$ is the $n$--fold iterated power of $2$.  They are representatives of the successive levels of the   Grzegorczyk hierarchy of primitive recursive functions --- see, for example, \cite{Rose}.    
 
\subsection{Background}  Other heavily distorted free subgroups of hyperbolic groups have been exhibited by Mitra~\cite{Mitra}:   for all $k$, he gives an example with  a free subgroup of  distortion like a $k$--fold iterated exponential function and, more extreme,  an example where the number of iterations grows like $\log n$. Barnard, the first author, and Dani developed Mitra's constructions into more explicit examples that are also $\textup{CAT}(-1)$  \cite{BBD}.  
 We are not aware of any example of a hyperbolic group with a  finite--rank free subgroup  of distortion exceeding that of our examples.   Indeed, we we are not even aware of a hyperbolic group with a \emph{finitely presented} subgroup of greater distortion.    The Rips construction, applied to a finitely presentable group with unsolvable word problem yields a hyperbolic (in fact, $C'(1/6)$ small--cancellation) group $G$ with a finitely generated subgroup $N$ such that $\Dist^G_N$ is not bounded above by any recursive function, but these $N$ are not finitely presentable.   (See  \cite[\S3.4]{AO}, \cite[Corollary~8.2]{Farb}, \cite[\S3, $3.K_3''$]{Gromov} and \cite{PittetThesis}.)



\bs

\subsection{An outline of our approach}  
Our groups $\Gamma_k$  are elaborations of the \emph{hydra groups} 
\begin{equation*} G_k \ = \ \langle \  a_1, \ldots, a_k,  t \  \mid \  t^{-1} {a_1} t=a_1, \ t^{-1}{a_i} t=a_i a_{i-1} \ (\forall i>1)   \ \rangle
\end{equation*}
explored by the second and third authors in \cite{DR} --- groups that are $\textup{CAT}(0)$, free--by--cyclic,  biautomatic, and can be presented with only one relator, and yet the subgroups  $H_k :=  \langle a_1t, \ldots, a_kt \rangle$ are free of rank $k$ and their  distortion grows like the $k$--th of Ackermann's functions: $\Dist^{G_k}_{H_k}  \simeq  A_k$.  

This extreme distortion stems from a phenomenon which can be described in terms of a  re--imagining  of Hercules' battle with the  Lernaean Hydra.  A \emph{hydra} is a positive word $w$ on the alphabet $a_1, a_2, \ldots$.  Hercules  removes the first letter and then the creature regenerates in that each remaining $a_i$ with $i > 1$  becomes $a_ia_{i-1}$ and each remaining $a_1$ is unaffected.  This repeats and Hercules triumphs when the hydra is  reduced to the empty word.  The number of steps is denoted $\H(w)$.  (Each step encompasses  the removal of the first letter and then regeneration.) 
For example, $\H({a_2}^3)=7$ ---
$${a_2}^3   \ \to \      (a_2a_1)^2  \ \to \      a_1 a_2  {a_1}^2  \  \to \     a_2 {a_1}^3 \  \to \    {a_1}^3    \  \to \  {a_1}^2  \  \to \  {a_1}   \  \to \   \varepsilon.$$
In \cite{DR} it is shown that Hercules will be victorious whatever  hydra he faces, but the number of strikes it takes can be huge: the functions $\mathcal{H}_k$ defined by $\mathcal{H}_k(n) = \H({a_k}^n)$ grow like Ackermann's functions: $\mathcal{H}_k \simeq A_k$.

The group $G_k$ is not hyperbolic because it has the subgroup   $\langle a_1, t \rangle \cong \Z^2$.  We obtain $\Gamma_k$ by  combining $G_k$ with another free--by--cyclic group, which is hyperbolic, in such a way that the interaction between the two occurs in a pair of relations that replace $t^{-1} {a_1} t=a_1$.     This will allow the hydra effect to persist in $\Gamma_k$.

Like $G_k$, the group $\Gamma_k$ is free--by--cyclic and will be presented as such.   We will prove it is hyperbolic by translating to an alternative presentation for which the associated Cayley 2--complex is $\textup{CAT}(0)$, and then we will argue that there are no isometrically embedded copies of $\Reals^2$ and appeal to the Flat Plane Theorem.    

Whilst we will not call on it in this paper (as we will give the translation between the presentations  explicitly), a result that lies behind how we came to our examples  is that if a 2--complex admits an $S^1$--valued Morse function all of whose   ascending and  descending links are trees, then its fundamental group is free--by--cyclic  \cite{BaBr}.   [The ascending  link for our examples is visible in Figure~\ref{Link} as the subgraph  made up of all edges connecting pairs of negative vertices.  The descending is that made up of all edges connecting pairs of positive vertices.  Both are trees.]

\ms

\subsection{The (lack of) upper bound on the distortion}

It seems likely that $\Dist^{\Gamma_k}_{\Lambda_k}  \simeq  A_k$, but we do not offer a proof that $\Dist^{\Gamma_k}_{\Lambda_k}  \preceq  A_k$.  
The proof in \cite{DR} that  $\Dist^{G_k}_{H_k} \preceq A_k$ in \cite{DR} may guide a proof that  $\Dist^{\Gamma_k}_{\Lambda_k}  \preceq  A_k$, but that proof is technical and how to carry it over to $\Dist^{\Gamma_k}_{\Lambda_k}$ is not readily apparent.  We are content to present here just the lower bound (which we believe is the more significant).

\subsection{The organisation of this article.}
In Section~\ref{our examples} we give two presentations for  $\Gamma_k$ and prove they are equivalent.  In Section~\ref{hyperbolicity} we use the link condition to show that the Cayley 2--complex associated to one of these presentations is  $\textup{CAT}(0)$ and then the Flat Plane Theorem   to establish hyperbolicity of $\Gamma_k$.  The other presentation  places $\Gamma_k$ in a class of free--by--cyclic groups which are the subject of Section~\ref{freeness and distortion}: we show that for $k \geq 2$, they contain free subgroups of rank $k+18$ and distortion $\simeq A_k$.   (In the case $k=1$, Theorem~\ref{Main} is elementary: take $\Gamma_1$ to be a free group and $\Lambda_1$ to be $\Gamma_1$.)



\section{Our examples} \label{our examples}

\subsection{A $\textup{CAT}(0)$ presentation for $\Gamma_k$.}   This presentation $P_k$ is well suited to establishing hyperbolicity (see Section~\ref{hyperbolicity}) --- 

\qquad generators: \qquad  $\alpha_1, \ldots, \alpha_k, \  \beta_1, \ldots, \beta_8,  \ \gamma_1, \ldots, \gamma_8, \ \sigma, \tau,$ 

\qquad relations: \vspace*{-5mm}
$$\begin{array}{rlr}
  \rule{0mm}{5mm} {\alpha_{i}}^{-1} \tau \alpha_{i}  \  = &   \!\! \alpha_{i-1}  \ \ (1 < i \leq  k), &    \\ 
  {\beta_{i}}^{-1} \tau \beta_{i} \   = &  \!\!  \beta_{i+1}  \ \ (1 \leq i \leq 7),   \quad  {\beta_{8}}  \tau  {\beta_{8}}^{-1} & \!\!\!\!  =  \   \beta_{1}, \\ 
  {\gamma_{i}}^{-1} \sigma \gamma_{i}  \ = & \!\!   \gamma_{i+1} \ \ (1 \leq i \leq 7),   \quad  {\gamma_{8}}  \sigma  {\gamma_{8}}^{-1} & \!\!\!\!   = \  \gamma_{1}, \\ 
  \gamma_3 \beta_5 \ = & \!\!  \beta_3 \gamma_5, \ \  \ \alpha_1 \gamma_1 \tau  \ = \ \tau \gamma_7 \alpha_1. \\
\end{array}$$ 

One can be encode $P_k$ in diagrams such as in Figure~\ref{LOTs} which gives the case $k=6$.    
Each edge in the  three labelled oriented trees (\emph{LOT}s --- see \cite{Howie}) encodes a  commutator relation ---  an edge labelled $y$ from a vertex labelled $x$ to a vertex labelled $z$ corresponds to a relation $y^{-1}xy=z$.   The square and hexagonal  2--cells represent the remaining two relations, $\gamma_3 \beta_5  =    \beta_3 \gamma_5$ and $\alpha_1 \gamma_1 \tau    =  \tau \gamma_7 \alpha_1.$  
 
 \begin{figure}[ht]
\psfrag{a}{\footnotesize{$\beta_1$}}
\psfrag{b}{\footnotesize{$\beta_2$}}
\psfrag{c}{\footnotesize{$\beta_3$}}
\psfrag{d}{\footnotesize{$\beta_4$}}
\psfrag{e}{\footnotesize{$\beta_5$}}
\psfrag{f}{\footnotesize{$\beta_6$}}
\psfrag{g}{\footnotesize{$\beta_7$}}
\psfrag{h}{\footnotesize{$\beta_8$}}
 \psfrag{s}{\footnotesize{$\sigma$}}
 \psfrag{t}{\footnotesize{$\tau$}}
 \psfrag{A}{\footnotesize{$\gamma_1$}}
\psfrag{B}{\footnotesize{$\gamma_2$}}
\psfrag{C}{\footnotesize{$\gamma_3$}}
\psfrag{D}{\footnotesize{$\gamma_4$}}
\psfrag{E}{\footnotesize{$\gamma_5$}}
\psfrag{F}{\footnotesize{$\gamma_6$}}
\psfrag{G}{\footnotesize{$\gamma_7$}}
\psfrag{H}{\footnotesize{$\gamma_8$}}
 \psfrag{1}{\footnotesize{$\alpha_1$}}
\psfrag{2}{\footnotesize{$\alpha_2$}}
\psfrag{3}{\footnotesize{$\alpha_3$}}
\psfrag{4}{\footnotesize{$\alpha_4$}}
\psfrag{5}{\footnotesize{$\alpha_5$}}
\psfrag{6}{\footnotesize{$\alpha_6$}}
 \centerline{\epsfig{file=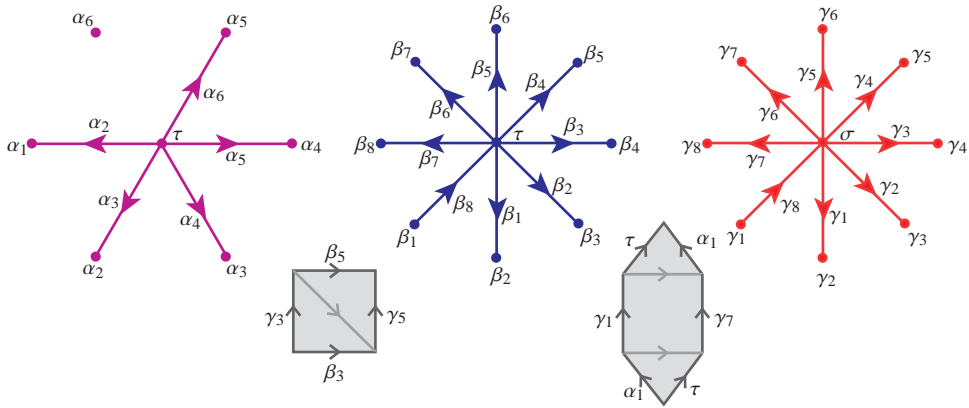}} \caption{The defining relations of the presentation $P_6$ for $\Gamma_6$ displayed as three  LOTs and two 2--cells.} \label{LOTs}
\end{figure}

 If one removes the $\alpha_i$ and all the relations in which they appear from  $P_k$, then one essentially gets groups  studied by Mecham \& Muckerjee in \cite{MeMu}.  These, in turn, are built from two copies of groups studied by Barnard and the first author in \cite{BaBr}.

\subsection{A presentation of $\Gamma_k$ as a  free--by--cyclic group.} \label{Q presentation}  This presentation $Q_k$  has --- 

\qquad  generators: \qquad $a_0, \ldots, a_k, \  b_1, \ldots, b_8,  \ c_1, \ldots, c_8, \   d, \ t,$

\qquad  relations:  
$$\begin{array}{rllll}
 t^{-1}a_it    & \!\!  =   \    \theta(a_i) & \quad 0 \leq i \leq k, \qquad &      t^{-1}c_it    &  \!\!  = \     w^{-1} \psi(c_i) w \qquad \ 1 \leq i \leq 8,     \\
   t^{-1}b_it    & \!\!   = \    \phi(b_i) & \quad 1 \leq i \leq 8, &     t^{-1}d t & \!\! = \    \psi(c_3) \psi(c_5)  \phi(b_5)^{-1}   \,  d \,  \phi(b_3),   \\
 \end{array}$$         
where $\theta$,  $\phi$ and $\psi$ are defined by  
\begin{align*}
 \theta(a_i)    \   = \ &   
         \begin{cases} 
                  \ u a_1 v    &     i=0, \\
         \ a_0    &     i=1, \\
         \ a_i a_{i-1}  &   1 < i \leq  k,  
         \end{cases}   \\
   \phi(b_i)    \   = \  &   
         \begin{cases} 
           \ b_i \ldots b_8   &  1 \leq i \leq 7, \\  
          \ {b_8}^{-1} \ldots {b_1}^{-1}b_8   &     i=8,            
         \end{cases} \\
   \psi(c_i)    \   = \  &   
         \begin{cases} 
           \   c_i \ldots c_8   &  1 \leq i \leq 7, \\  
           \   {c_8}^{-1} \ldots {c_1}^{-1}c_8    &     i=8,            
         \end{cases}           
\end{align*}      
and   $u   :=    t^{-k} {c_7}^{-1} c_3 d {b_3}^{-1}  t^{k}$,  $v   :=   t^{-k} {b_3} d^{-1} {c_3}^{-1}  {c_1}   t^{k}$, and 
$w  :=   \psi({c_3})^{-1}  \,  t^{-1}  b_3 d^{-1} t$. 


Three remarks:
\begin{enumerate}
\item   $u$, $v$, and $w$ represent elements of the subgroup $\langle b_1, \ldots, b_8, \, c_1, \ldots, c_8, \, d \rangle$.
\item   Both $\phi$ and $\psi$ define automorphisms of  $F(b_1, \ldots, b_8)$  and $F(c_1, \ldots, c_8)$, respectively, as would $\theta$ for  $F(a_0, \ldots, a_k)$ were $\theta(a_0)$ equal to $a_1$ rather than $u a_1 v$.  
\item The action of $t$ by conjugation on $$F(a_0, a_1, \ldots, a_k, \  b_1, \ldots, b_8,  \ c_1, \ldots, c_8, \   d)$$ apparent in the presentation $Q_k$ is an automorphism because there are sequences of Nielsen transformations taking  
\begin{align*}
(a_0, & a_1,  \ldots, a_k, \ b_1, \ldots, b_8, \  c_1, \ldots, c_8,  \ d) \\
 &  \to \ (a_1, t^{-1}a_1t,   \ldots, t^{-1}a_kt, \  b_1, \ldots, b_8, \ c_1, \ldots, c_8,  \ d) \\
 &  \to \ (t^{-1}a_0t, t^{-1}a_1t,   \ldots, t^{-1}a_kt, \  b_1, \ldots, b_8,  \ c_1, \ldots, c_8, \ d) \\
 &  \to \ (t^{-1}a_0t, t^{-1}a_1t,   \ldots, t^{-1}a_kt, \  b_1, \ldots, b_8, \  \psi(c_1), \ldots, \psi(c_8), \ d) \\ 
 & \to \ (t^{-1}a_0t, t^{-1}a_1t,   \ldots, t^{-1}a_kt, \ b_1, \ldots, b_8, \ \psi(c_1), \ldots, \psi(c_8), \ t^{-1}dt) \\
 & \to \ (t^{-1}a_0t, t^{-1}a_1t,   \ldots, t^{-1}a_kt, \  b_1, \ldots, b_8,  \ t^{-1} c_1 t, \ldots, t^{-1} c_8 t, \ t^{-1}dt) \\
 & \to \ (t^{-1}a_0t, t^{-1}a_1t,   \ldots, t^{-1}a_kt, \  t^{-1}b_1t, \ldots, t^{-1}b_8t, \  t^{-1} c_1 t, \ldots, t^{-1} c_8 t, \ t^{-1}dt).
 \end{align*}  
 \end{enumerate}

The subgroup $\Lambda_k$ of Theorem~\ref{Main} will be $$\langle \, a_0t, \ldots, a_kt, \, b_1, \ldots, b_8, \, c_1, \ldots, c_8, \, d \, \rangle.$$ 

\subsection{The equivalence of the presentations.}

\begin{prop} \label{equiv presentations}
$P_k$ and $Q_k$ present the same groups.   
\end{prop}

Before we prove this, let $P'_k$ be the presentation obtained from $P_k$ by removing all the generators $\alpha_i$ and all the relations in which they occur.  
Likewise, let $Q'_k$ be that obtained from $Q_k$ by removing all the generators $a_i$ and all the relations in which they occur.  

\begin{lemma} \label{smaller equiv pres}
The groups presented by $P'_k$ and $Q'_k$ are isomorphic via 
$$\begin{array}{ll}
 \tau \mapsto t^{-1}, \qquad & 
\beta_i \mapsto 
   \begin{cases} 
   t^{-1}b_i & 1 \leq i \leq 7, \\
   b_8t^{-1} & i=8, \\
   \end{cases}   
 \\    
 \sigma \mapsto s^{-1}, \qquad & 
\gamma_i \mapsto 
   \begin{cases} 
   s^{-1}c_i & 1 \leq i \leq 7, \\
   c_8s^{-1} & i=8, \\
   \end{cases}   
    \end{array}$$
where $s =  c_3 d {b_3}^{-1} t$.  
\end{lemma}

\begin{proof}
As  the reader can verify, $\tau \mapsto t^{-1}$, \, $\beta_i \mapsto t^{-1}b_i$ for $1 \leq i \leq 7$, and $\beta_8 \mapsto b_8t^{-1}$   defines an isomorphism 
\begin{align*} 
 \langle \ & \beta_1, \ldots, \beta_8, \tau \ \mid \ {\beta_{i}}^{-1} \tau \beta_{i} \,  = \, \beta_{i+1}  \   (1 \leq i \leq 7),   \quad  {\beta_{8}}  \tau  {\beta_{8}}^{-1} \,  =  \,   \beta_{1} \rangle \\
& \qquad  \to \  F(b_1, \ldots, b_8) \rtimes_{\phi} \Z \ = \ \langle \, b_1, \ldots, b_8, t \, \mid \, t^{-1}b_i t = \phi(b_i)  \  \ (1 \leq i \leq 8)  \, \rangle. 
\end{align*} 
 Likewise,  
 \begin{align*} 
\langle \ & \gamma_1, \ldots, \gamma_8, \sigma \ \mid \ {\gamma_{i}}^{-1} \sigma \gamma_{i} \,  = \, \gamma_{i+1}    \textup{ for } 1 \leq i \leq 7,   \   {\gamma_{8}} \sigma  {\gamma_{8}}^{-1} \,  =  \,   \gamma_{1} \, \rangle \\
& \qquad  \cong  \ \langle \, c_1, \ldots, c_8, s \, \mid \, s^{-1}c_i s = \psi(c_i)   \textup{ for }  1 \leq i \leq 8 \, \rangle, 
\end{align*}   
via $\sigma \mapsto s^{-1}$, \, $\gamma_i \mapsto s^{-1} c_i$ for $1 \leq i \leq 7$, and $\gamma_8 \mapsto c_8s^{-1}$.
 
These maps translate the relation $\gamma_3 \beta_5  =    \beta_3 \gamma_5$ to $s^{-1} c_3 t^{-1} b_5 = t^{-1} b_3  s^{-1} c_5$, which becomes ${b_5}^{-1} t d =  {c_5}^{-1}  {c_3}^{-1} d {b_3}^{-1} t$ on  substituting for  $s =  c_3 d  {b_3}^{-1} t$ and simplifying.  Inverting, we find this is equivalent to  $t {b_5}^{-1} t  d  \  = \  {c_5}^{-1} {c_3}^{-1} d {b_3}^{-1} t$, which rearranges to  $t^{-1}d t   =  (t^{-1} c_3 c_5 {b_5}^{-1} t) \,  d \,  (t^{-1} b_3 t)$. 
 
 Now ${c_3}^{-1} s t^{-1} b_3 d^{-1} =1$ since $s= c_3 d {b_3}^{-1} t$.  So the relation $s^{-1} c_i s = \psi(c_i)$ is equivalent to 
$t^{-1} c_i t =  w^{-1} \psi(c_i) w$ since 
\begin{align*}
t^{-1} c_i t  & \ = \ t^{-1} \, (d {b_3}^{-1} t s^{-1} c_3) \, c_i \, ({c_3}^{-1} s t^{-1} b_3 d^{-1}) \, t    \\
& = \  (t^{-1} dt) \, (t^{-1} {b_3}^{-1} t) \,   (s^{-1} c_3 s ) \,   (s^{-1} c_i s ) \, (s^{-1} {c_3}^{-1} s ) \,  (t^{-1} b_3  t) \,  (t^{-1} d^{-1}  t) \\
& = \  w^{-1}    \psi(c_i)  w.   
\end{align*}
\end{proof}

Inductively define words $u_i$ and $v_i$ for $i \geq 0$ by
$$\begin{array}{llll}
 u_0  & \!\! = \  \alpha_k, \qquad u_{i+1} &  \!\! = \  {u_i}^{-1} t^{-1} u_i & (i \geq 0), \\
 v_0  & \!\! = \  a_k, \qquad v_{i+1} &  \!\! = \   {v_i}^{-1} t^{-1}  v_i t  & (i \geq 0). \\
\end{array}$$
 The following observation from \cite{DR} can be proved by inducting on $i$. 
 
\begin{lemma} \label{word change}
On substituting an $a_k$ for each $\alpha_k$ in $u_i$,  the words $u_i $ and $t^{i-1}v_it^{-i}$ become freely equal for all $i \geq 1$.  
\end{lemma}

\begin{proof}[Proof of Proposition~\ref{equiv presentations}]
By Lemma~\ref{smaller equiv pres} there is a sequence of Tietze moves  carrying the subpresentation $P'_k$ of $P_k$ to $Q'_k$ and 
the remaining relations (those involving the $\alpha_i$) to 
$$\alpha_1 s^{-1} c_1  t^{-1} \  = \ t^{-1} s^{-1} c_7 \alpha_1,     \qquad {\alpha_{i}}^{-1}  t^{-1} \alpha_{i}  \  = \    \!\! \alpha_{i-1},  \ \ 1 < i \leq  k.$$ 
A sequence of Tietze moves eliminating $\alpha_1$,  \ldots, $\alpha_{k-1}$ transforms this family to the single relation $$u_{k-1} s^{-1} c_1  t^{-1} \  = \ t^{-1} s^{-1} c_7 u_{k-1}.$$  
Now substitute an $a_k$ for each $\alpha_k$, then by Lemma~\ref{word change}, this relation is equivalent to 
$$( t^{k-2} v_{k-1} t^{-(k-1)}) s^{-1} c_1  t^{-1} \  = \ t^{-1} s^{-1} c_7 \ ( t^{k-2} v_{k-1} t^{-(k-1)}),$$
which becomes 
$$t^{-1} v_{k-1} t \  = \  (t^{-(k-1)}  {c_7}^{-1} s t^{k-1}) \ v_{k-1} \ (t^{-(k-1)}  s^{-1} c_1 t^{k-1})$$
on conjugating by $t^{k-1}$ and rearranging.  A sequence of Tietze moves introducing $a_{k-1}$, \ldots, $a_1$ expands this to the family of relations
  $$ t^{-1} a_1 t  =  t^{-(k-1)} {c_7}^{-1}s t^{k-1}  \, a_1 \,   t^{-(k-1)} s^{-1} c_1 t^{k-1},   \qquad t^{-1}a_it   =  a_i a_{i-1} , \   1 < i \leq  k.$$ 

The first of these relations becomes $ t^{-1} a_1 t  = a_0$ when we introduce $a_0$ together with the new relation $$a_0 \ = \ t^{-(k-1)} {c_7}^{-1}s t^{k-1}  \, a_1 \,   t^{-(k-1)} s^{-1} c_1 t^{k-1},$$  which  becomes $t^{-1} a_0 t = u a_1 v$ on conjugating by $t$ and eliminating the $s$ and $s^{-1}$ using $s =  c_3 d {b_3}^{-1} t$. 
 \end{proof}

\section{Hyperbolicity}  \label{hyperbolicity}


 We establish hyperbolicity using  techniques employed in \cite{BaBr} and \cite{MeMu}. 


Consider  the presentation 2--complex $K_k$ for  $P_k$ assembled from Euclidean unit--squares associated to each of the  defining relations with the single exception of $\alpha_1 \gamma_1 \tau    =  \tau \gamma_7 \alpha_1$ for which we use a Euclidean hexagon made from one unit square and two equilateral triangles as shown in Figure~\ref{LOTs}.

 \begin{figure}[ht]
 \psfrag{s+}{\footnotesize{$\sigma^+$}}
 \psfrag{s-}{\footnotesize{$\sigma^-$}}
 \psfrag{t+}{\footnotesize{$\tau^+$}}
 \psfrag{t-}{\footnotesize{$\tau^-$}}
\psfrag{1+}{\footnotesize{$\alpha^+_1$}}
\psfrag{2+}{\footnotesize{$\alpha^+_2$}}
\psfrag{3+}{\footnotesize{$\alpha^+_3$}}
\psfrag{4+}{\footnotesize{$\alpha^+_4$}}
\psfrag{5+}{\footnotesize{$\alpha^+_5$}}
\psfrag{6+}{\footnotesize{$\alpha^+_6$}}
\psfrag{1-}{\footnotesize{$\alpha^-_1$}}
\psfrag{2-}{\footnotesize{$\alpha^-_2$}}
\psfrag{3-}{\footnotesize{$\alpha^-_3$}}
\psfrag{4-}{\footnotesize{$\alpha^-_4$}}
\psfrag{5-}{\footnotesize{$\alpha^-_5$}}
\psfrag{6-}{\footnotesize{$\alpha^-_6$}}
\psfrag{a+}{\footnotesize{$\beta^+_1$}}
\psfrag{b+}{\footnotesize{$\beta^+_2$}}
\psfrag{c+}{\footnotesize{$\beta^+_3$}}
\psfrag{d+}{\footnotesize{$\beta^+_4$}}
\psfrag{e+}{\footnotesize{$\beta^+_5$}}
\psfrag{f+}{\footnotesize{$\beta^+_6$}}
\psfrag{g+}{\footnotesize{$\beta^+_7$}}
\psfrag{h+}{\footnotesize{$\beta^+_8$}}
\psfrag{a-}{\footnotesize{$\beta^-_1$}}
\psfrag{b-}{\footnotesize{$\beta^-_2$}}
\psfrag{c-}{\footnotesize{$\beta^-_3$}}
\psfrag{d-}{\footnotesize{$\beta^-_4$}}
\psfrag{e-}{\footnotesize{$\beta^-_5$}}
\psfrag{f-}{\footnotesize{$\beta^-_6$}}
\psfrag{g-}{\footnotesize{$\beta^-_7$}}
\psfrag{h-}{\footnotesize{$\beta^-_8$}}
\psfrag{A+}{\footnotesize{$\gamma^+_1$}}
\psfrag{B+}{\footnotesize{$\gamma^+_2$}}
\psfrag{C+}{\footnotesize{$\gamma^+_3$}}
\psfrag{D+}{\footnotesize{$\gamma^+_4$}}
\psfrag{E+}{\footnotesize{$\gamma^+_5$}}
\psfrag{F+}{\footnotesize{$\gamma^+_6$}}
\psfrag{G+}{\footnotesize{$\gamma^+_7$}}
\psfrag{H+}{\footnotesize{$\gamma^+_8$}}
\psfrag{A-}{\footnotesize{$\gamma^-_1$}}
\psfrag{B-}{\footnotesize{$\gamma^-_2$}}
\psfrag{C-}{\footnotesize{$\gamma^-_3$}}
\psfrag{D-}{\footnotesize{$\gamma^-_4$}}
\psfrag{E-}{\footnotesize{$\gamma^-_5$}}
\psfrag{F-}{\footnotesize{$\gamma^-_6$}}
\psfrag{G-}{\footnotesize{$\gamma^-_7$}}
\psfrag{H-}{\footnotesize{$\gamma^-_8$}}
\centerline{\epsfig{file=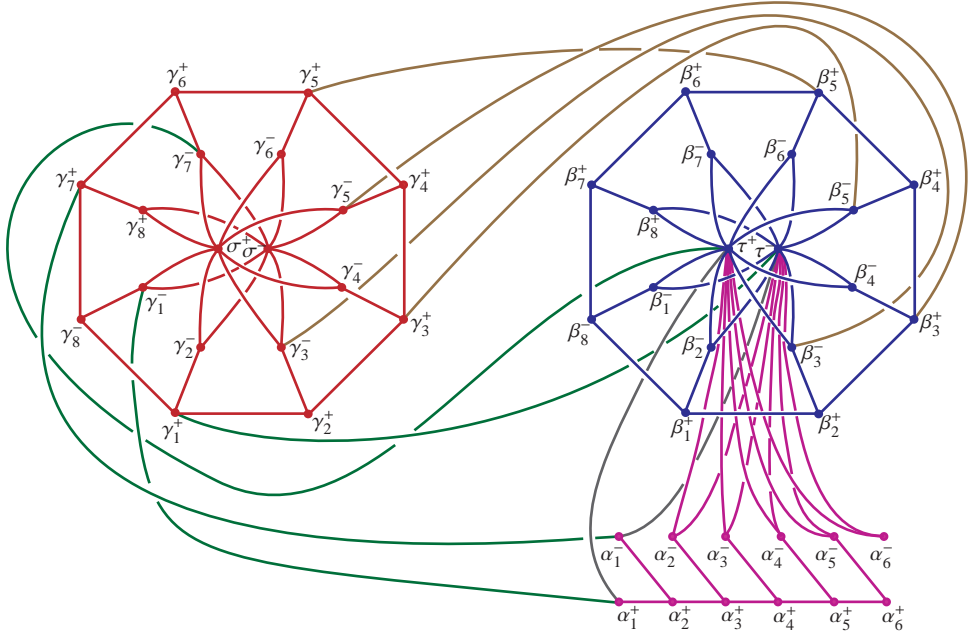}} \caption{The link of the vertex in the presentation 2--complex associated to the presentation $P_6$  for $\Gamma_6$ given in Section~\ref{our examples}.
The  two  grey edges have length $\pi/3$, the four green edges have length $5\pi/6$, and all other edges have length $\pi/2$.} \label{Link}
\end{figure}

The link in the case $k=6$ is shown in Figure~\ref{Link}.   All edges have length $\pi/2$ apart from the edges  $\tau^+  \mbox{---} \alpha_1^{+}$ and $\tau^-  \mbox{---}  \alpha_1^{-}$ (shown in grey), which have length $\pi/3$, and the edges from $\gamma_1^+  \mbox{---}  \tau^-$, $\gamma_7^+  \mbox{---}  \alpha_1^-$, $\alpha_1^+  \mbox{---}  \gamma_1^-$, and $\tau^+  \mbox{---}  \gamma_7^+$ (shown in green), which have length $5\pi/6$.
Inspecting the link we see that any simple loop in the graph has length at least $2 \pi$ (separately considering the cases of monochrome and multicoloured simple loops in Figure~\ref{LOTs} helps to check this)  --- that is, for all $k \geq 1$, the link is \emph{large}.
So $K_k$ satisfies the \emph{link condition} (see \cite{BrH}) and  its  universal cover $\widetilde{K_k}$ is therefore a $\textup{CAT}(0)$--space.  

To show that $\Gamma_k$ is hyperbolic we will show that $\widetilde{K_k}$ contains no subspace isometric to  $\mathbb{E}^2$ and then appeal to the Flat Plane Theorem of \cite{Bridson11, Gromov4}.  The link of a vertex in any isometric copy of $\mathbb{E}^2$ in  $\widetilde{K_k}$ would appear as a simple loop of length $2 \pi$ in the link.  But inspecting the link, we find that no edges of length $\pi/3$ or  $5\pi/6$ (the grey and green edges) occur in a simple loop of length $2 \pi$.  Next one can check  the edges  $\gamma_3^+  \mbox{---}  \beta_5^-$, $\gamma_3^-  \mbox{---}  \beta_3^-$, $\gamma_5^-  \mbox{---}  \beta_3^+$ and $\gamma_5^+  \mbox{---}  \beta_5^+$ (the brown edges in the figure) do not occur in a simple loop of length $2 \pi$.  Then it becomes evident that edges  occurring in simple loops of length $2 \pi$ are precisely the edges  
$$\begin{array}{llll}
\tau^+  \mbox{---}  \alpha_i^-,  &  \tau^-  \mbox{---}  \alpha_i^-  \  \  (2 \leq i \leq 7),  \\
\tau^+  \mbox{---}  \beta_i^-,  &  \tau^-  \mbox{---}  \beta_i^-  \  \  (1 \leq i \leq 7),   &  \quad  \tau^+  \mbox{---}  \beta_8^+, &  \tau^-  \mbox{---}  \beta_8^+, \\
\sigma^+  \mbox{---}  \gamma_i^-,  &  \sigma^-  \mbox{---}  \gamma_i^-  \  \  (1 \leq i \leq 7),   &  \quad  \sigma^+  \mbox{---}  \gamma_8^+, &  \sigma^-  \mbox{---}  \gamma_8^+. \\
 \end{array}$$ 
So every corner of every 2--cell in an isometrically embedded copy of $\mathbb{E}^2$ must give rise to one of the edges in this list.  But, looking at the defining relations, we see that no  2--cell in $\widetilde{K_k}$ has this property.  Therefore there are no such  $\mathbb{E}^2$ and  $\Gamma_k$ is hyperbolic by the Flat Plane Theorem.


\section{\texorpdfstring{Freeness and distortion}{Freeness  and distortion}}  \label{freeness and distortion}

\subsection{A family of free--by--cyclic groups}

Fix an integer $l \geq 1$, words $u$ and $v$ on $b_1, \ldots, b_l$, and an automorphism $\phi$ of $F(b_1, \ldots, b_l)$.  Then, for $k \geq 1$, define $$\Psi_k  \ := \ F(a_0, \ldots, a_k, b_1, \ldots, b_l) \rtimes_{\theta} \Z$$ where $\theta$ is the automorphism of $F(a_0, \ldots, a_k, b_1, \ldots, b_l)$ whose restriction to $F(b_1, \ldots, b_l)$ is $\phi$ and   
$$\theta(a_i)  \ = \  \begin{cases} 
         \ u a_1 v   &  i=0, \\  
         \ a_0    &     i=1, \\
         \ a_i a_{i-1}  &   1 < i \leq  k. 
         \end{cases}   $$
 Let $t$ denote a generator of the $\Z$--factor, so $t^{-1} a_i t = \theta(a_i)$ and $t^{-1} b_j t = \theta(b_j)$ for all $i$ and $j$.


The presentation $Q_k$  in Section~\ref{Q presentation} shows  $\Gamma_k$ is an example of such a $\Psi_k$.  

Our aim in the remainder of this section is to  establish ---

\begin{prop} \label{distortion and freeness prop}
 The subgroup $$\Lambda_k \ := \   \langle a_0t,   \ldots, a_kt, b_1, \ldots, b_l \rangle$$ of $\Psi_k$ is free of rank $k+l+1$ and $\Dist^{\Psi_k}_{\Lambda_k} \simeq A_k$.   
\end{prop}

\subsection{Towards a lower bound on distortion}

In what follows, when, for a word $u = u(a_0, \ldots, a_k, b_1, \ldots, b_l)$, we refer to $\theta^{r}(u)$, we mean the freely reduced word that equals $\theta^{r}(u)$  in $F(a_0, \ldots, a_k, b_1, \ldots, b_l)$.

The extreme  distortion in the hydra groups of \cite{DR} stemmed from the battle between Hercules and the hydra that we described in Section~\ref{intro}.  When studying $\Psi_k$ we will need the following more elaborate version of that battle.   A \emph{hydra} is  now a word on $a_0, a_1, \ldots, a_k, b_1, \ldots, b_l$ in which the $a_i$ only appear with positive exponents.  As before, Hercules fights a hydra by removing the first letter.  But in this version, the hydra only regenerates after an $a_i$ is removed, and that regeneration is:  each remaining $a_i$ and ${b_j}^{\pm 1}$ becomes $\theta(a_i)$    and $\theta({b_j}^{\pm 1})$, respectively.      Again, we consider Hercules victorious if, on sufficient repetition, the hydra is reduced to the empty word.  

Reprising the example from Section~\ref{intro}, Hercules defeats ${a_2}^3$  as follows: 
\begin{align*}
{a_2}^3   &  \  \to \      (a_2a_1)^2  \ \to \      a_0 a_2  a_1 a_0  \  \to \     a_2  a_1 a_0 ua_1v \  \to \    a_0 ua_1v   \theta(u) a_0 \theta(v)   \\    \ &  \to \    a_0 \theta(v)   \theta^2(u) u a_1  v \theta^2(v)   \  \to \    a_0  \theta(v) \theta^3(v)   \  \to \   \varepsilon.
\end{align*}
--- here, the steps in which Hercules removes a $b_j$ are not shown; the arrows indicate the progression from when an $a_i$ is about to be removed to when an $a_i$ next appears at the front of the word or the hydra becomes the empty word.  

The key point is that $a_0$ and the $b_j$ play no essential role in this battle; if we removed all $b_j$ and replaced all $a_0$ by $a_1$, we would have a battle of the original form.   Thus we have the following lemma.   [Recall that $\H(u)$  denotes the duration of the battle (of the original type from Section~\ref{intro} and \cite{DR}) against the hydra $u$.]   

\begin{lemma}  \label{elaborated hydra}
Hercules wins against all hydra $w$ and, in the battle, the number of times he removes an $a_i$  equals $\H(\bar{w})$ where $\bar{w}$ is the word obtained from $w$ by removing all   ${b_j}^{\pm 1}$ and replacing all $a_0$ by $a_1$.
\end{lemma}

Consideration of the original  battle between Hercules and the hydra led to the result that, for all $k,n \geq 1$, there is a positive word $u_{k,n} = u_{k,n}(a_1t, \ldots, a_kt)$ of length $\H_k(n)$ that equals ${a_k}^n t^{\H_k(n)}$ in $G_k$.  (This is Lemma~5.1 in \cite{DR}.)   The reason is that the pairing off of a $t$ with an initial $a_i$ in a positive word on $a_1, \ldots, a_k$ corresponds to a decapitation, and the conjugation by $t$ that moves that $t$ into place from the right--hand end causes \emph{regeneration} for the remainder of the word.  For example $\H_2(3)=7$ and
\begin{align*}
{a_2}^3 t^{7}   &  \  =  \    (a_2t) \   t^{-1} {a_2}^2 t  \  t^6  \ = \      (a_2t) \    (a_2a_1)^2  \, t^6    \  = \    (a_2t) (a_2t)  \   t^{-1}   a_1{a_2a_1}  t  \ t^5   \\  \ &  = \     \cdots \ = \  (a_2t)(a_2t)(a_1t)(a_2t)(a_1t)(a_1t)(a_1t) \ = \  u_{2,3}.
\end{align*}

In the corresponding calculation for $\Psi_k$,   only the $a_i$ get paired with $t$, and on  each  of the $\H_k(n)$ times  that happens, the subsequent conjugation by $t$ can increase length by a  factor  $C$ which depends only on $\phi$, $\ell(u)$ and $\ell(v)$.   So ---
\begin{lemma}  \label{big triangle}
There exists $C>0$ such that for all $k,n \geq 1$, there is a  word $\hat{u}_{k,n} =  \hat{u}_{k,n}(a_0t, \ldots, a_kt, b_1, \ldots, b_l)$ that equals ${a_k}^n t^{\H_k(n)}$ in $\Psi_k$ and has the properties that $$\H_k(n) \ \leq \  \ell(\hat{u}_{k,n})  \ \leq \  C^{\mbox{$\H_k(n)$}}n$$ and all the $(a_it)$ it contains have positive exponents.
 \end{lemma}
 
This and our next two lemmas will be components of a calculation that will yield Proposition~\ref{lower bound on distortion} (the analogue of  Proposition~5.2 in \cite{DR}), which will be  the key to establishing a lower bound on the distortion of $\Lambda_k$ in $\Psi_k$.

A simple calculation yields ---

\begin{lemma}  \label{calculation1}   $t^{-m}a_1t^{m+1} = \tau_m$ in $\Psi_k$ for all $m \geq 2$ where 
\begin{equation*}
\tau_m   \  :=  \    
\begin{cases}
a_0t & \text{for } m = 1 \\  
\phi^{m-2}(u) \cdots  \phi^{2}(u)   u  \   (a_1t)  \ \phi(v) \phi^3(v) \cdots \phi^{m-1}(v) & \text{for even } m \geq 2 \\
\phi^{m-2}(u) \cdots  \phi^{3}(u) \phi(u)     \   (a_0t)  \ \phi^2(v) \phi^4(v) \cdots \phi^{m-1}(v) & \text{for odd } m >  2.
\end{cases}
 \end{equation*}
 \end{lemma}
 
 This combines with  $$t^{-1}(a_1t^{-1})^n  \   =  \ (t^{-1}a_1t^2) \, ( t^{-3} a_1 t^4) \, ( t^{-5} a_1 t^6) \cdots ( t^{1-2n} a_1 t^{2n}) t^{-2n-1}$$ to give ---
 
\begin{lemma}  \label{calculation2}
There is a constant $C>0$, depending only on $\phi$, $\ell(u)$ and $\ell(v)$, such that for all $n \geq 0$, there is a word $v_n = v_n(a_0t, a_1t, b_1, \ldots, b_l)$ such that   $t^{-1} (a_1 t^{-1})^n = v_n t^{-2n-1}$ in $\Psi_k$, the number of $(a_1t)$ contained in $v_n$  is $n$ and all have positive exponent, and $n \leq \ell(v_n) \leq C^n$.  
  \end{lemma}

\begin{prop} \label{lower bound on distortion}
For all $k \geq 2$ and $n \geq 1$, there is a reduced word of length at least $2 \H_k(n) +3$ on   $a_0t, a_1t, \ldots, a_kt, b_1, \ldots, b_l$, that equals ${a_k}^n a_2  ta_1 {a_2}^{-1} {a_k}^{-n}$ in  $\Psi_k$.  
\end{prop}

\begin{proof} After rewriting the relation $t^{-1} a_2 t  = a_2 a_1$  as ${a_2}^{-1} t a_2 = t {a_1}^{-1}$,  we see that   ${a_2}^{-1}  t^{\H_k(n)} a_2 = (t {a_1}^{-1})^{\H_k(n)}$.  So 
$${a_k}^n a_2 =     \hat{u}_{k,n} a_2  (t {a_1}^{-1})^{-\H_k(n)}$$
for $\hat{u}_{k,n}$ as in  Lemma~\ref{big triangle}.
This gives the first of the equalities   
\begin{align*}
{a_k}^n a_2  ta_1 {a_2}^{-1} {a_k}^{-n}  & \ = \    \hat{u}_{k,n} a_2 (t {a_1}^{-1})^{- \H_k(n)} ta_1 (t {a_1}^{-1})^{\H_k(n)}  {a_2}^{-1}  {\mbox{$\hat{u}_{k,n}$}}^{-1}  \\
& \ = \    \hat{u}_{k,n} (a_2t) \  t^{-1}  (t {a_1}^{-1})^{- \H_k(n)} ta_1 (t {a_1}^{-1})^{\H_k(n)} t  \  {(a_2t)}^{-1}  {\mbox{$\hat{u}_{k,n}$}}^{-1}  \\
& \ = \   \hat{u}_{k,n} (a_2t) \   v_{\H_k(n)} t^{-2n-1}  ta_1  t^{2n+1} {v_{\H_k(n)}}^{-1}    \  {(a_2t)}^{-1} {\mbox{$\hat{u}_{k,n}$}}^{-1} \\
& \ = \   \hat{u}_{k,n} (a_2t) \   v_{\H_k(n)}  \tau_{2n} {v_{\H_k(n)}}^{-1}    \  {(a_2t)}^{-1}  {\mbox{$\hat{u}_{k,n}$}}^{-1}.
\end{align*} 
The second is a free equality and  the third and fourth are applications of Lemmas~\ref{calculation2} and \ref{calculation1}, respectively.

This calculation arrives at a word on  $a_0t, a_1t, \ldots, a_kt, b_1, \ldots, b_l$, that equals ${a_k}^n a_2  ta_1 {a_2}^{-1} {a_k}^{-n}$ in  $\Psi_k$.  This word may not be freely reduced, but if we delete all the ${b_j}^{\pm 1}$ it contains, replace all ${a_0}^{\pm 1}$ by ${a_1}^{\pm 1}$, and then freely reduce (i.e. cancel away all $(a_it)^{\pm 1}(a_it)^{\mp 1}$ subwords), we  get $u_{k,n} \,  (a_2t) \,  (a_1 t) \, ({a_2}t)^{-1} \, {u_{k,n}}^{-1}$, which has length $2 \mathcal{H}_k(n) + 3$.     
\end{proof}

\subsection{Freeness and rank}  

The result of this section is: 

\begin{prop} \label{freeness}  
The subgroup $\Lambda_k$ is free of rank $k+l +1$. 
\end{prop}



It will be convenient to prove more.  In the special case where $w$ represents the identity, the following proposition tells us that there are  no non--trivial relations between  $a_0t$, \ldots, $a_kt, b_1, \ldots, b_l$ and so establishes  Proposition~\ref{freeness}.

\begin{prop} \label{a power of t prop}
If $w = w(a_0t, \ldots, a_kt, b_1, \ldots, b_l)$ represents an element of the subgroup $\langle t \rangle$ in $\Psi_k$, then $w$ freely equals the empty word. 
\end{prop}

We begin  with an observation on how the groups $\Psi_k$ nest.

\begin{lemma} \label{nesting}
For $1\leq i \leq k$, the canonical homomorphism $\Psi_i \to \Psi_k$ is an inclusion.
\end{lemma}

\begin{proof}
The free--by--cyclic normal forms --- a reduced word on $a_0, \ldots, a_k, b_1, \ldots, b_l$ times a power of $t$ --- of an element of $\Psi_i$ and its image in $\Psi_k$ are the same. 
\end{proof}

We will  prove Proposition~\ref{a power of t prop} by induction, but first we give a corollary which will be useful in the induction step.  
We emphasise that when we say that $v(a_0t, \ldots, a_kt, b_1, \ldots, b_l)$ is  \emph{freely reduced} in the following, we mean that there are no $(a_it)^{\pm 1}(a_it)^{\mp 1}$ or ${b_j}^{\pm 1}{b_j}^{\mp 1}$ subwords.  

\begin{cor} \label{positivity for prop}
Suppose $v(a_0t, \ldots, a_kt, b_1, \ldots, b_l)$ is a freely reduced word equalling  $\hat{v}t^s$ in $\Psi_k$ where $s \in \Z$ and $\hat{v} = \hat{v}(a_0, \ldots, a_k, b_1, \ldots, b_l)$ is a word in which all the $a_i$ that occur have positive exponents.  Then all the $(a_it)$ in $v$ have positive exponents.   
\end{cor} 

\begin{proof}
The hydra battle described prior to Lemma~\ref{elaborated hydra}, played out against $\hat{v}(a_0, \ldots, a_k, b_1, \ldots, b_l)$ gives a  word $v' = v'(a_0t, \ldots, a_kt, b_1, \ldots, b_l)$ and an integer $s'$ such that  $v'= \hat{v}t^{s'}$ in $\Psi_k$.  Moreover, the exponents of all the $(a_it)$ in $v'$  are positive.  Now,   $v^{-1}v' \in \langle t \rangle$ since $\hat{v} = vt^{-s} = v' t^{-s'}$, and so  $v$ and $v'$ are freely equal by Proposition~\ref{a power of t prop}.  Therefore the exponents of all the $(a_it)$ in $v$ are positive.   
\end{proof}

\begin{proof}[Proof of Proposition~\ref{a power of t prop}.]
We induct on $k$.   For the base case of $k=1$,  notice that defining $\bar{a}_0 := a_0 t$ and $\bar{a}_1 := a_1 t$, we can transform the presentation  
$$\langle \,  a_0, a_1, b_1, \ldots, b_l,  t \,  \mid \, t^{-1} {a_0} t=ua_1v, \  t^{-1} {a_1} t=a_0, \ t^{-1}{b_j} t=\phi(b_j) \  \forall j    \, \rangle$$
for $\Psi_1$ to 
$$\langle \,  \bar{a}_0, \bar{a}_1, b_1, \ldots, b_l,  t \,  \mid \, t^{-1} {\bar{a}_0} t=u\bar{a}_1\phi(v), \  t^{-1} {\bar{a}_1} t=\bar{a}_0, \ t^{-1}{b_j} t=\phi(b_j) \  \forall j    \, \rangle,$$ which is an alternative means of expressing $\Psi_1$ as a free--by--cyclic group from which  
the result  is evident.

 For the induction step, we consider a freely reduced word $w = w(a_0t, \ldots, a_kt, b_1, \ldots, b_l)$ representing an element of  $\langle t \rangle$ in $\Psi_k$ where $k \geq 2$.    If  no $(a_kt)^{\pm 1}$ are present in $w$ we can deduce from the induction hypothesis and Lemma~\ref{nesting} that $w$ freely reduces to the empty word.  For the remainder of our proof we suppose there are $(a_kt)^{\pm 1}$ present, and we seek a contradiction.  
 
 Consider shuffling  the $t^{\pm 1}$ to the start of $w$ using the defining relations --- replacing each $a_i$ and $b_j$  passed by a $t^{\pm 1}$ with $\theta^{\pm 1}(a_i)$ and $\theta^{\pm 1}(b_j)$, respectively.  The result will be a power of $t$ times a word on $a_0, \ldots, a_k, b_1, \ldots, b_l$ which freely reduces to the empty word.  Such is $\theta$, no $a_k$ are created or destroyed in this process of shuffling the $t^{\pm 1}$.  So there is some expression  $w_0{(a_kt)}^{\pm 1} u {(a_kt)}^{\mp 1}w_1$  for $w$ such that $u = u(a_0t, \ldots, a_{k-1}t,b_1, \ldots, b_l)$ and the  ${a_k}^{\pm 1}$ and ${a_k}^{\mp 1}$ in the ${(a_kt)}^{\pm 1}$ and ${(a_kt)}^{\mp 1}$  buttressing $u$  cancel after the shuffling and free reduction.   

We will address first the case  $w = w_0{(a_k t)}^{- 1} u (a_kt) w_1$.  Break down the shuffling process by first shuffling the $t^{\pm 1}$ out of  $w_0$, $u$ and $w_1$, and then carrying the resulting powers to the front of the word: 
\begin{align*}
w  & \ = \ w_0 \, (a_kt)^{-1} \, u \, (a_k t) \,  w_1    
    \ \to \  t^{r_0} \,  \hat{w}_0 \, (a_kt)^{-1} \, t^{r} \, \hat{u} \, (a_kt) \, t^{r_1} \, \hat{w}_1   \\ 
   & \ \to \   t^{r_0+r +r_1}  \, \theta^{r+r_1} \, (\hat{w}_0) \, \theta^{r+r_1+1}({a_k}^{- 1}) \,  \theta^{r_1+1}(\hat{u}) \, \theta^{r_1+1}(a_k) \, \hat{w}_1 
\end{align*}   
where $r_0, r, r_1 \in \Z$ and  
\begin{align*}
\hat{w}_0 & \  = \ \hat{w}_0(a_0, \ldots, a_k, b_1, \ldots, b_l), \\ 
\hat{u} & \  = \ \hat{u}(a_0, \ldots, a_{k-1}, b_1, \ldots, b_l), \\ 
\hat{w}_1 & \ = \ \hat{w}_1(a_0, \ldots, a_k, b_1, \ldots, b_l)
\end{align*} 
are words such that $t^{r_0} \hat{w}_0 = w_0$, $t^{r} \hat{u} = u$ and $t^{r_1} \hat{w}_1 = w_1$ in $\Psi_k$.   
When we expand $\theta^{r+r_1}({a_k}^{- 1})$ and $\theta^{r_1+1}({a_k}) $ as words on $a_0, \ldots, a_k$, the former ends with an ${a_k}^{- 1}$ which must cancel with the ${a_k}$ at the start of the latter.  So $\theta^{r_1+1}(\hat{u})$, and therefore $\hat{u}$,  freely equal  the empty word.  So $u$  represents an element of $\langle t \rangle$ and, by induction hypothesis, freely reduces to the empty word, contrary to the initial assumption that  $w(a_0t, \ldots a_kt, b_1, \ldots, b_l)$ is reduced.  
  
In the case $w = w_0 (a_kt) u (a_kt)^{- 1}w_1$, the shuffling process is
\begin{align*}
w  & \ = \ w_0 \, (a_kt) \, u \, (a_k t)^{-1} \,  w_1    
    \ \to \  t^{r_0} \,  \hat{w}_0 \, (a_kt) \, t^{r} \, \hat{u} \, (a_kt)^{-1} \, t^{r_1} \, \hat{w}_1   \\ 
   & \ \to \   t^{r_0+r +r_1}  \, \theta^{r+r_1} \, (\hat{w}_0) \, \theta^{r+r_1}({a_k}) \,  \theta^{r_1-1}(\hat{u}) \, \theta^{r_1}({a_k}^{- 1}) \, \hat{w}_1 
\end{align*}   
where $t^{r_0} \hat{w}_0 = w_0$, $t^{r} \hat{u} = u$ and $t^{r_1} \hat{w}_1 = w_1$ in $\Psi_k$, as before. 
The first and last letters  of   $\theta^{r+r_1}({a_k}) \,  \theta^{r_1-1}(\hat{u}) \, \theta^{r_1}({a_k}^{- 1})$  are $a_k$ and ${a_k}^{-1}$ which cancel, so this subword must freely reduce to the empty word.  So $\theta^{r}({a_k}) \, \theta^{-1}(\hat{u}) \,  {a_k}^{- 1}$ also freely reduces to the empty word --- that is, $\theta^{r+1}(a_k) \, \hat{u}$ freely equals $a_k a_{k-1}$.  

If $r=0$ then this says that  $\hat{u}$ freely equals the empty word and, as before, we have a contradiction.  Suppose $r >0$.  Then $\hat{u}^{-1} = (a_k a_{k-1})^{-1} \theta^{r+1}(a_k)$ would be a positive word on $a_0, \ldots, a_{k-1}$ were we to remove all the ${b_1}^{\pm 1}, \ldots, {b_l}^{\pm 1}$ it contains.  
So, as $\hat{u}^{-1}t^{-r} = u^{-1}$, Corollary~\ref{positivity for prop} applies and tells us that $u^{-1}$ would be a positive word were we to remove all the ${b_1}^{\pm 1}, \ldots, {b_l}^{\pm 1}$ it contains.  But $r$ is the exponent sum of the  $(a_0t)^{\pm 1}, \ldots, (a_{k-1}t)^{\pm 1}$ in $u$, and so we deduce the contradiction $r \leq 0$.  Finally we note that the case $r<0$ also leads to a contradiction because if we replace  $w$ by  $w^{-1}$ it becomes the case  $r>0$. 
\end{proof}

\subsection{Conclusion} We deduce from Proposition~\ref{a power of t prop} that the word posited in Proposition~\ref{lower bound on distortion}  is the \emph{unique} reduced word on $a_0t,   \ldots, a_kt, b_1, \ldots, b_l$ that equals  ${a_k}^n a_2  ta_1 {a_2}^{-1} {a_k}^{-n}$  in $\Psi_k$.  This establishes that  $\Dist_{\Lambda_k}^{\Psi_k} \succeq \H_k$ for all $k \geq 2$.  So, by Proposition~1.2 in \cite{DR}, which  says that  $\H_k \simeq A_k$ for all $k \geq 1$, we have $\Dist_{\Lambda_k}^{\Psi_k} \succeq A_k$ for all $k \geq 2$.  Added to Proposition~\ref{freeness}, this completes  the proof of Proposition~\ref{distortion and freeness prop}.

Proposition~\ref{distortion and freeness prop} applies to the subgroup $$\langle \, a_0t, \ldots, a_kt, \, b_1, \ldots, b_8, \, c_1, \ldots, c_8, \, d \, \rangle$$ of $\Gamma_k$ (presented as $Q_k$ of Section~\ref{our examples}) and so, as we established $\Gamma_k$ to be hyperbolic in Section~\ref{hyperbolicity}, Theorem~\ref{Main} is proved.

 






\bibliographystyle{plain}
\bibliography{$HOME/Dropbox/Bibliographies/bibli}

\ni
\small{\textsc{Noel Brady} \rule{0mm}{6mm} \\
Department of Mathematics, Physical Sciences Center, 601 Elm Ave, University of Oklahoma,  Norman,   OK      73019, USA \\ \texttt{nbrady@math.ou.edu}, \
\href{http://aftermath.math.ou.edu/~nbrady/}{http://aftermath.math.ou.edu/$\sim$nbrady/}

\ni
\small{\textsc{Will Dison} \rule{0mm}{6mm} \\
Bank of England, Threadneedle Street, London, EC2R 8AH, UK \\ \texttt{william.dison@gmail.com}, \
\href{http://www.maths.bris.ac.uk/~mawjd/}{http://www.maths.bris.ac.uk/$\sim$mawjd/}

\ni  \textsc{Timothy R.\ Riley} \rule{0mm}{6mm} \\
Department of Mathematics, 310 Malott Hall,  Cornell University, Ithaca, NY 14853, USA \\ \texttt{tim.riley@math.cornell.edu}, \
\href{http://www.math.cornell.edu/~riley/}{http://www.math.cornell.edu/$\sim$riley/}}
\end{document}